\documentclass[11pt]{article}%
\usepackage[all]{xy}
\usepackage{amssymb,latexsym}
\usepackage{graphicx}
\usepackage{amsmath,amsthm}
\usepackage{enumerate}
\usepackage{url}
\newcommand{\Pp}{{\mathbb P}}

\newcommand{\OO}{{\mathcal O}}

\newcommand{\Xc}{{\mathcal X}}

\newtheorem{thm}{Theorem}[section]

\begin{document}


\title{A note on the paper "ON CURVES WITH SPLIT JACOBIANS"}
\author{Sajad Salami \\
	\\
	Inst\'{i}tuto de Matem\'{a}tica e Estat\'{i}stica \\
	Universidade Estadual do Rio do Janeiro, Brazil\\
	Email: sajad.salami@ime.uerj.br}
\maketitle

\begin{abstract}
	In \cite{Yamau}, without giving a detailed proof, Yamauchi provided a formula to calculate the genus of a certain family of smooth complete intersection algebraic curves. That formula is used extensively in \cite{BSHSH} to study on the  algebraic curves for which their Jacobian has superelliptic components.
In this note, we determine  the correct version of the genus formula  with an algebraic proof.
Then,  	we   show that the  formula given in \cite{Yamau} works only under certain conditions.	
\end{abstract}


\section{Introduction and main result}

Let $ r \geq 2$ and $ s \geq 1$ be arbitrary integers. 
Let $k$ be a field and denote by ${\bar k}$ an algebraically closed field containing $k$.
Fix a system of coordinates
$x, y_1, \cdots, y_s, w, z$ on the projective space $\Pp_{\bar k}^{s+2}$.
Let $\Xc_{r,s}$  denotes the algebraic curve  defined over $k$  in  $\Pp_{\bar k}^{s+2}$ by the following equations,
\begin{equation}
\label{eq1}
\begin{split}
zw & = c_0x^2+c_1xw+c_2w^2,\\
y_1^r& =h_1(z,w):=z^r+c_{1,1}z^{r-1}w+ \cdots + c_{r-1,1}zw^{r-1}+w^r,\\
\vdots & \vdots\\
y_s^r& =h_1(z,w):=z^r+c_{1,s}z^{r-1}w+ \cdots + c_{r-1,s}zw^{r-1}+w^r,  
\end{split}
\end{equation}
where $c_\ell\in k$, $\ell=0,1,2,$ and $c_{i,j}\in k$ for $i=1,\cdots, r$, and $j=1,\cdots, s$.

Without giving a comprehensive proof, in Proposition 4.1 (i) of \cite{Yamau}, Yamauchi stated that
the  genus  $g(\Xc_{r,s})$ of the curve $\Xc_{r,s}$ is equal to
$(r - 1)(r s\cdot 2^{s-1} - 2^s + 1)$ if it is smooth and $c_0 \not = 0$.
His  genus  formula  is used extensively  in Section 4 of  \cite{BSHSH}  to provide an   necessary and sufficient condition in terms of $r$ and $s$ such that the  Jacobian of  $g(\Xc_{r,s})$ to decompose as Jacobian of superelliptic curves.

We note that  the Jacobian matrix of $\Xc_{r,s}$ can not have a full rank if the characteristic of $k$  divides  $2 r$. 
Thus, in this note,   we assume that  $k$ is a field of characteristic not dividing  $2 r$.
we shall to provide a correct formula for the genus of  $\Xc_{r,s}$ as stated in the following theorem. 

\begin{thm}
	Assume that the curve $\Xc_{r,s}$ is smooth and   $c_0 \not = 0$. Then, its  genus is equal to 
	\begin{equation}
	\label{eq2}
	g(\Xc_{r,s})= 1+ r^s(s(r-1) -1).
	\end{equation}
\end{thm}
\begin{proof}
	Let $\omega_{\Xc_{r,s}}$  be  the canonical sheaf of $\Xc_{r,s}$.
	By the exercise (II.8.4.e) in \cite{hart}, it is isomorphic to 
	$ \OO (rs+2-(s+2)-1)=\OO(s(r-1) -1).$
	Hence, we have
	$$\deg (\omega_{\Xc_{r,s}})= (s(r-1) -1) \cdot \deg(\Xc_{r,s}).$$
	Using the classical version of Bezout's theorem, see Proposition 8.4 in \cite{fulton1} or Example 1 in page 198 of \cite{shaf1},
	the degree of  $\Xc_{r,s}$ over   $\bar k$ is equal to $\deg(\Xc_{r,s})=2\cdot r^s$.
	Thus, we get that  $\deg (\omega_{\Xc_{r,s}})= 2\cdot r^s\cdot (s(r-1) -1).$
	As a consequence of the Riemann-Roch theorem, 	it is well known that the degree of canonical sheaf of 
	any algebraic curve of genus $g$ is equal to $2g-2$.
	For instance, see example 1.3.3 in chapter IV of \cite{hart}.
	Therefore, we have $2g(\Xc_{r,s}) -2 =2\cdot r^s\cdot(s(r-1) -1)$ that leads to the desired formula.
\end{proof}

We remark that the genus formula given in \cite{Yamau} and here coincide  in two cases,
say when $r=2$ and any $s\geq 1$ as well as
$r\geq 2$ and $s=1$. Hence,  Theorem 4.2 in  \cite{Yamau}  and  Theorem 4.3 in  \cite{BSHSH}  and its consequences are true only in the above mentioned two cases and are wrong in other cases.


\begin{thebibliography}{99}
	\bibitem {BSHSH}{Beshaj, T.,  Shaska, T., and  Shor, C.} On Jacobians of curves with superelliptic components,
	In: Riemann and Klein Surfaces, Automorphisms, Symmetries and Moduli Spaces 629, 1-14 (2014).
	
	
	
	\bibitem {fulton1}{Fulton, W.} Intersection theory,  Second edition.
	Ergebnisse der Mathematik und ihrer Grenzgebiete. 3. Folge. A Series of
	Modern Surveys in Mathematics [Results in Mathematics and Related
	Areas. 3rd Series. A Series of Modern Surveys in Mathematics],
	2. Springer-Verlag, Berlin, (1998).
	
	\bibitem {hart} Hartshorne, R., Algebraic geometry,
	Graduate Texts in Mathematics, Vol. 52, Springer-Verlag, New york (1977).
	

	
	\bibitem {shaf1}{Shafarevich, I. R.} Basic algebraic geometry,  Translated from
	the Russian by K. A. Hirsch.  Revised printing of Grundlehren der
	mathematischen Wissenschaften, Vol.  213, 1974.  Springer Study Edition.
	Springer-Verlag, Berlin-New York, 1977.
	
	
	
	
	
	\bibitem {Yamau}{Yamauchi, T.} On curves with split Jacobians, Comm. Algebra 36, no. 4, 1419–1425  (2008).
	
	
	
	
	

\end{thebibliography}
\end{document}